\documentclass{amsart}
\usepackage{amssymb, amsmath, mathrsfs, verbatim, url, bbm,marvosym}
\usepackage[all,cmtip]{xy}

\theoremstyle{plain}
\newtheorem{theorem}{Theorem}[section]
\newtheorem{lemma}[theorem]{Lemma}
\newtheorem{proposition}[theorem]{Proposition}
\newtheorem{corollary}[theorem]{Corollary}

\theoremstyle{definition}
\newtheorem{definition}[theorem]{Definition}
\newtheorem{question}[theorem]{Question}

\newcommand{\cl}{\mathsf{cl}}
\newcommand{\diam}{\mathsf{diam}}
\newcommand{\re}{\upharpoonright}
\newcommand{\ZFC}{\mathsf{ZFC}}
\newcommand{\SU}{\mathsf{SU}}
\newcommand{\Diff}{\mathsf{D}}
\newcommand{\Cof}{\mathsf{Cof}}
\newcommand{\Fin}{\mathsf{Fin}}
\newcommand{\Ss}{\mathcal{S}}
\newcommand{\BB}{\mathcal{B}}

\newcommand{\TT}{\mathcal{T}}

\newcommand{\FF}{\mathcal{F}}
\newcommand{\GG}{\mathcal{G}}

\newcommand{\XX}{\mathcal{X}}
\newcommand{\YY}{\mathcal{Y}}
\newcommand{\II}{\mathcal{I}}

\newcommand{\PP}{\mathcal{P}}
\newcommand{\RR}{\mathcal{R}}
\newcommand{\ZZ}{\mathcal{Z}}
\newcommand{\CCC}{\mathbf{C}}
\newcommand{\DDD}{\mathbf{D}}
\newcommand{\SSS}{\mathbf{S}}
\newcommand{\TTT}{\mathbf{T}}
\newcommand{\PPP}{\mathbf{P}}
\newcommand{\QQQ}{\mathbf{Q}}

\begin{document}

\title{On Borel semifilters}

\author{Andrea Medini}
\address{Kurt G\"odel Research Center for Mathematical Logic
\newline\indent University of Vienna
\newline\indent W\"ahringer Stra{\ss}e 25
\newline\indent A-1090 Wien, Austria}
\email{andrea.medini@univie.ac.at}
\urladdr{http://www.logic.univie.ac.at/\~{}medinia2/}

\keywords{Semifilter, Borel, homogeneous, filter, h-homogeneous, Wadge.}

\thanks{The author acknowledges the support of the FWF grant M 1851-N35.}

\date{June 7, 2016}

\begin{abstract}
Building on work of van Engelen and van Mill, we show that a zero-dimensional Borel space is homeomorphic to a semifilter if and only if it is homogeneous and not locally compact.
Under $\mathbf{\Sigma}^1_1$-Determinacy, this result extends to all analytic and coanalytic spaces.
\end{abstract}

\maketitle

\section{Introduction}

Throughout this paper, $\Omega$ will denote a countably infinite set. We will denote by $\PP(\Omega)$ the collection of all subsets of $\Omega$.
Define $\Fin(\Omega)=\{x\subseteq\Omega:x\text{ is finite}\}$ and $\Cof(\Omega)=\{x\subseteq\Omega:\Omega\setminus x\text{ is finite}\}$.
Also define $\Fin=\Fin(\omega)$ and $\Cof=\Cof(\omega)$.

A collection $\XX\subseteq\PP(\Omega)$ is \emph{upward-closed} if and only if $y\supseteq x\in\XX$ implies $y\in\XX$ for all $x,y\in\PP(\Omega)$.
We will write $x\subseteq^\ast y$ to mean that $x\setminus y$ is finite, and we will write $x=^\ast y$ to mean that $x\subseteq^\ast y$ and $y\subseteq^\ast x$.
A collection $\XX\subseteq\PP(\Omega)$ is \emph{closed under finite modifications} if and only if $x\in\XX$ and $y=^\ast x$ implies $y\in\XX$ for all $x,y\in\PP(\Omega)$.

A \emph{semifilter on $\Omega$} is a collection $\Ss\subseteq\PP(\Omega)$ that satisfies the following conditions.
\begin{itemize}
\item $\varnothing\notin\Ss$ and $\Omega\in\Ss$.
\item $\Ss$ is closed under finite modifications.
\item $\Ss$ is upward-closed.
\end{itemize}
All semifilters are assumed to be on $\omega$ unless we explicitly say otherwise.

The notion of semifilter is a natural weakening of the notion of filter, and it has found applications in several areas of mathematics (see \cite{banakhzdomskyy}).

Throughout this paper, we will freely identify any collection $\XX\subseteq\PP(\Omega)$ with the subspace of $2^{\Omega}$ consisting of the characteristic functions of elements of $\XX$. In particular, every semifilter will inherit the subspace topology from $2^\omega$.
Notice that $\Cof\subseteq\Ss$ and $\Fin\cap\Ss=\varnothing$ for every semifilter $\Ss$. In particular, every semifilter is dense in $2^\omega$ and not locally compact.

By \emph{space} we will always mean separable metrizable topological space.
Recall that a space $X$ is \emph{homogeneous} if and only if for every $x,y\in X$ there exists a homeomorphism $h:X\longrightarrow X$ such that $h(x)=y$.
This is a classical notion in topology, which has been studied in depth. In particular, in his remarkable doctoral thesis \cite{vanengelent}, van Engelen obtained a complete classification of the homogeneous zero-dimensional Borel spaces.
Recall that a space $X$ is an \emph{absolute Borel set} (or simply \emph{Borel}) if every homeomorphic copy of $X$ in any space $Z$ is a Borel subspace of $Z$.
Using Lavrentiev's Theorem (see \cite[Theorem 3.9]{kechris}), one can show that a space is Borel if and only if it is homeomorphic to a Borel subspace of some completely metrizable space.

The following is our main result, and it is a consequence of Corollary \ref{semifilterhomogeneous}, Theorem \ref{indelta} and Theorem \ref{notindelta}.
In Section 14, assuming $\mathbf{\Sigma}^1_1$-Determinacy, we will show that this result extends to all analytic and coanalytic spaces.
\begin{theorem}\label{main}
Let $X$ be a zero-dimensional Borel space. Then the following conditions are equivalent.
\begin{itemize}
\item $X$ is homeomorphic to a semifilter.
\item $X$ is homogeneous and not locally compact.
\end{itemize}
\end{theorem}

\section{Topological preliminaries}

Given spaces $X$ and $Y$, we will write $X\approx Y$ to mean that $X$ is homeomorphic to $Y$. Let $\CCC=2^\omega$ denote the Cantor set, $\PPP=\omega^\omega$ denote Baire space, and $\QQQ$ denote the space of rationals. Given $s\in 2^{<\omega}$, we will use the notation $[s]=\{x\in\CCC:s\subseteq x\}$.

We will assume some familiarity with the theory of Borel sets, and in particular with the notions of $\mathbf{\Sigma}^0_\xi$, $\mathbf{\Pi}^0_\xi$ and $\mathbf{\Delta}^0_\xi$ subset of a space for $1\leq\xi<\omega_1$ (see for example \cite[Section 11.B]{kechris}). For brevity, we will simply write \emph{complete} when we mean completely metrizable. It is well-known that a subspace of a complete space is complete if and only if it is $\mathbf{\Pi}^0_2$ (see \cite[Theorem 3.11]{kechris}).

We will also assume that the reader is comfortable with the basic theory of analytic and coanalytic subsets of a complete space.
As in \cite[page 315]{kechris}, we will say that a space is \emph{analytic} (respectively \emph{coanalytic}) if it is homeomorphic to an analytic (respectively coanalytic) subspace of some complete space (see \cite[Section 4]{medinizdomskyy} for a more detailed treatment).

A subset of a space is \emph{clopen} if and only if it is closed and open. A space is \emph{zero-dimensional} if and only if it has a base consisting of clopen sets. It is well-known that a space is zero-dimensional if and only if it is homemorphic to a subspace of $\CCC$ (see for example \cite[Corollary 1.5.7]{vanmilli}). Given a topological property $\PP$, it will be convenient to say that a space $X$ is \emph{nowhere $\PP$} if and only if $X$ is non-empty and no non-empty open subspace of $X$ is $\PP$.

A space is \emph{dense in itself} if and only if it is non-empty and it has no isolated points. We will be using freely the following classical characterizations of $\QQQ$, $\CCC$, and $\PPP$ (see \cite[Theorem 2.4.1]{vanengelent}, \cite[Theorem 2.1.1]{vanengelent}, and \cite[Theorem 2.3.1]{vanengelent} respectively). Let $X$ be a zero-dimensional space.
\begin{itemize}
\item $X\approx\QQQ$ if and only if $X$ is a dense in itself countable space.
\item $X\approx\CCC$ if and only if $X$ is a dense in itself compact zero-dimensional space.
\item $X\approx\PPP$ if and only if $X$ is a complete nowhere compact zero-dimensional space.
\end{itemize}

In the rest of this article, we will often exclude locally compact spaces from our treatment, as they constitute the trivial case. In fact, it is easy to see that if $X$ is a zero-dimensional homogeneous locally compact space then either $X$ is discrete, $X\approx\CCC$, or $X\approx\omega\times\CCC$.

A space $X$ is \emph{first category} if and only if $X=\bigcup_{n\in\omega}X_n$, where the closure of each $X_n$ has empty interior.
A space $X$ is \emph{Baire} if and only if $\bigcap_{n\in\omega}U_n$ is dense in $X$ whenever each $U_n$ is an open dense subset of $X$.
We will be using freely the following well-known facts (see \cite[1.12.1]{vanengelent} and \cite[1.12.2]{vanengelent} respectively).
\begin{itemize}
\item Every homogeneous space is either first category or Baire.
\item If a Borel space is Baire then it contains a dense complete subspace.
\end{itemize}

A space $X$ is \emph{strongly homogeneous} (or \emph{h-homogeneous}) if and only if $U\approx X$ for every non-empty clopen subspace $U$ of $X$. It is well-known that every zero-dimensional strongly homogeneous space is homogeneous (see \cite[1.9.1]{vanengelent} or \cite[Proposition 3.32]{medinit}). The following is a special case of \cite[Theorem 2.4]{terada} (see also \cite[Theorem 2 and Appendix A]{medinih} or \cite[Theorem 3.2 and Appendix B]{medinit}).
\begin{lemma}[Terada]\label{pibase}
Let $X$ be a non-compact space, and assume that $X$ has a base $\BB$ consisting of clopen sets such that $U\approx X$ for every $U\in\BB$. Then $X$ is strongly homogeneous.
\end{lemma}

\section{Filters, semiideals, and ideals}

A collection $\XX$ is \emph{closed under finite intersections} if and only if $x\cap y\in\XX$ for all $x,y\in\XX$.
A \emph{filter on $\Omega$} is a semifilter on $\Omega$ that is closed under finite intersections.

The following result is \cite[Theorem 3.4]{vanengeleni}, and it gives a purely topological characterization of filters among the zero-dimensional Borel spaces, in the same spirit as Theorem \ref{main}. In fact, it is the result which inspired this entire article. Our phrasing is slightly different from the original, but the discussion in the remainder of this section should clarify all the possible confusion.
\begin{theorem}[van Engelen]
Let $X$ be a zero-dimensional Borel space. Then the following conditions are equivalent.
\begin{itemize}
\item $X$ is homeomorphic to a filter.
\item $X$ is homogeneous, first category, homeomorphic to $X^2$, and not locally compact.
\end{itemize}
\end{theorem}

A collection $\XX\subseteq\PP(\Omega)$ is \emph{downward-closed} if and only if $x\subseteq y\in\XX$ implies $x\in\XX$ for all $x,y\in\PP(\Omega)$.
A \emph{semiideal on $\Omega$} is a collection $\RR\subseteq\PP(\Omega)$ that satisfies the following conditions.
\begin{itemize}
\item $\varnothing\in\RR$ and $\Omega\notin\RR$.
\item $\RR$ is closed under finite modifications.
\item $\RR$ is downward-closed.
\end{itemize}
A collection $\XX$ is \emph{closed under finite unions} if and only if $x\cup y\in\XX$ for all $x,y\in\XX$.
An \emph{ideal on $\Omega$} is a semiideal on $\Omega$ that is closed under finite unions.
All filters, semiideals and ideals are assumed to be on $\omega$ unless we explicitly say otherwise.

Next we will show that, from the topological point of view, semifilters (respectively filters) are indistinguishable from semiideals (respectively ideals).
Given any $F\subseteq\omega$, define $h_F:\CCC\longrightarrow \CCC$ by setting 
$$
\left.
\begin{array}{lcl}
& & h_F(x)(n)= \left\{
\begin{array}{ll}
1-x(n) & \textrm{if }n\in F,\\
x(n) & \textrm{if }n\in\omega\setminus F.
\end{array}
\right.
\end{array}
\right.
$$
It is easy to check that $h_F$ is a homeomorphism for every $F\subseteq\omega$.

Throughout this article, we will let $c=h_\omega$ denote the complement function.
Given any $\XX\subseteq\PP(\omega)$, it is trivial to check that $\XX$ is a semifilter (respectively a semiideal) if and only if $c[\XX]$ is a semiideal (respectively a semifilter).
Similarly, one sees that $\XX$ is a filter (respectively an ideal) if and only if $c[\XX]$ is an ideal (respectively a filter).
Since $\XX\approx c[\XX]$, this means that every result about the topology of semifilters (respectively filters) immediately translates to a result about semiideals (respectively ideals), and viceversa.

As an application of this principle, one can see that every filter is a topological group.
In fact, every ideal is a topological subgroup of $\CCC$ under the operation of coordinatewise addition modulo $2$,
and any space that is homeomorphic to a topological group is itself a topological group.
In particular, every filter is homogeneous. Corollary \ref{semifilterhomogeneous} shows that this hold for semifilters as well, but the proof is considerably more involved.
For an example of a semifilter that is not a topological group, consider the spaces $\SSS$ and $\TTT$ as described in Propositions \ref{S} and \ref{T}.
These spaces are not topological groups by \cite[Corollary 3.6.6]{vanengelent}.

We conclude this section by remarking that many authors (including van Engelen in \cite{vanengeleni}) give a more general notion of filter than the one we gave above.
The most general notion possible seems to be the following. Define a \emph{prefilter on $\Omega$} to be a collection of subsets of $\Omega$ that is upward-closed and closed under finite intersections.
The next proposition, which can be safely assumed to be folklore, shows that our definition of filter does not result in any substantial loss of generality. Given a collection $\XX$ consisting of subsets of $\omega$ and $\Omega\subseteq\omega$, define $\XX\re\Omega=\{X\cap\Omega:X\in\XX\}$.
\begin{proposition}
Let $\GG$ be an infinite prefilter on $\omega$. Then either $\GG\approx\CCC$ or $\GG\approx\FF$ for some filter $\FF$.
\end{proposition}
\begin{proof}
Let $\Omega=\omega\setminus\bigcap\GG$, and observe that $\Omega$ is infinite because $\GG$ is infinite.
Notice that $\GG\re\Omega$ is a prefilter on $\Omega$. First assume that $\varnothing\in\GG\re\Omega$.
This means that $\bigcap\GG=\omega\setminus\Omega\in\GG$, hence $\GG=\{X\subseteq\omega:\bigcap\GG\subseteq X\}\approx\CCC$.

Now assume that $\varnothing\notin\GG\re\Omega$. We claim that $\GG\re\Omega$ is in fact a filter on $\Omega$.
In order to prove this claim, it will be enough to show that $\Cof(\Omega)\subseteq\GG\re\Omega$. So let $F$ be a finite subset of $\Omega$. Since $\Omega=\omega\setminus\bigcap\GG$ and $\GG$ is closed under finite intersections, there must be $X\in\GG$ such that $X\subseteq\omega\setminus F$.
It follows that $\omega\setminus F\in\GG$, hence $\Omega\setminus F\in\GG\re\Omega$. Finally, it is straightforward to check that $\GG\approx\GG\re\Omega$.
\end{proof}

\section{Every semifilter is homogeneous}

The following result is a fundamental tool for constructing homogeneous spaces, and it first appeared as \cite[Lemma 2.1]{vanmillh} (see also \cite[Lemma 1.9.1]{vanmilli}, or \cite[Theorem 3]{medvedev} for a more general result). Corollary \ref{finitemodificationshomogeneous} is essentially the same as \cite[Lemma 2]{medinivanmillzdomskyy}.

\begin{theorem}[van Mill]\label{homogeneitylemma}
Assume that $X$ is a zero-dimensional space, and fix a compatible metric on $X$. Let $x,y\in X$. Suppose that, for every $\varepsilon >0$,
there exist clopen neighborhoods $U$ and $V$ of $x$ and $y$ respectively such that $\diam(U)<\varepsilon$, $\diam(V)<\varepsilon$ and $U\approx V$.
Then there exists a homeomorphism $h:X\longrightarrow X$ such that $h(x)=y$.
\end{theorem}

\begin{corollary}\label{finitemodificationshomogeneous}
Let $X$ be a subspace of $\CCC$ that is closed under finite modifications. Then $X$ is homogeneous.
\end{corollary}
\begin{proof}
Fix a compatible metric on $\CCC$. Recall the definition of $h_F$ from Section 3.
Observe that $h_F[X]=X$ whenever $F$ is finite, because $X$ is closed under finite modifications.
In order to verify the hypotheses of Theorem \ref{homogeneitylemma}, fix $x,y\in X$ and $\varepsilon >0$.
Start by choosing $m\in\omega$ large enough so that $\diam([x\re m])<\varepsilon$ and $\diam([y\re m])<\varepsilon$.
Let $U=[x\re m]\cap X$ and $V=[y\re m]\cap X$. It is clear that $U$ and $V$ are clopen neighborhoods in $X$ of $x$ and $y$ respectively.
To see that $U\approx V$, simply observe that $h_F[U]=V$, where $F=\{n<m:x(n)\neq y(n)\}$. 
\end{proof}

\begin{corollary}\label{semifilterhomogeneous}
If $\Ss$ is a semifilter then $\Ss$ is homogeneous.
\end{corollary}

\begin{corollary}\label{borelsemifilterstronglyhomogeneous}
If $\Ss$ is a Borel semifilter then $\Ss$ is strongly homogeneous.
\end{corollary}

\begin{proof}
Simply apply Corollary \ref{semifilterhomogeneous} and \cite[Corollary 4.4.6]{vanengelent}.
\end{proof}

\section{Some concrete examples}

In this section, if $X$ is one of the notable spaces $\QQQ$, $\QQQ\times\CCC$, $\SSS$ or $\TTT$ (see below), we will exhibit a semifilter that is homeomorphic to $X$. The first two cases are given by the following trivial proposition, where the desired semifilter will actually be a filter.

\begin{proposition}\label{QandQC}
If $X=\QQQ$ or $X=\QQQ\times\CCC$ then there exists a filter $\FF$ such that $\FF\approx X$.
\end{proposition}
\begin{proof}
If $X=\QQQ$, simply let $\FF=\Cof$. Now assume that $X=\QQQ\times\CCC$. Fix infinite sets $\Omega_1$ and $\Omega_2$ such that $\Omega_1\cup\Omega_2=\omega$ and $\Omega_1\cap\Omega_2=\varnothing$. Define $\FF=\{x\subseteq\omega:\Omega_1\subseteq^\ast x\}$, and observe that $\FF$ is a filter. Since
$$
\FF=\{x_1\cup x_2:x_1\in\Cof(\Omega_1)\text{ and }x_2\subseteq\Omega_2\},
$$
one sees that $\FF\approx\Cof(\Omega_1)\times 2^{\Omega_2}\approx X$.
\end{proof}

The remainder of this section will not be needed later, but it might useful for a better understanding of Section 7. Furthermore, the proofs of Propositions \ref{S} and \ref{T} yield concrete descriptions of $\SSS$ and $\TTT$ that are as nice as possible from the combinatorial point of view.

The spaces $\SSS$ and $\TTT$ were introduced respectively by van Mill (in \cite{vanmillc}) and by van Douwen (unpublished). In hindsight, they are the first non-trivial step in the classification of the homogeneous zero-dimensional spaces in $\Delta$ (see Sections 6 and 7). We will not give the original definitions of $\SSS$ and $\TTT$, but use the following characterizations instead (see \cite[Section 5]{vanmillc} and \cite[Appendix]{vanengelenvanmill} respectively).

\begin{theorem}[van Mill; van Douwen]\label{SandT}
Let $X$ be a zero-dimensional space.
\begin{itemize}
\item $X\approx\SSS$ if and only if $X$ is the union of a complete subspace and a $\sigma$-compact subspace, $X$ is nowhere $\sigma$-compact, and $X$ is nowhere the union of a complete and a countable subspace.
\item $X\approx\TTT$ if and only if $X$ is the union of a complete subspace and a countable subspace, $X$ is nowhere $\sigma$-compact, and $X$ is nowhere complete.
\end{itemize}
\end{theorem}

\begin{proposition}\label{S}
There exists a semifilter $\Ss$ such that $\Ss\approx\SSS$.
\end{proposition}
\begin{proof}
Fix infinite sets $\Omega_1$ and $\Omega_2$ such that $\Omega_1\cup\Omega_2=\omega$ and $\Omega_1\cap\Omega_2=\varnothing$. Also fix an infinite $\Omega\subseteq\Omega_2$ such that $\Omega_2\setminus\Omega$ is infinite. Define
$$
\Ss=\{x_1\cup x_2:x_1\subseteq\Omega_1\text{, }x_2\subseteq\Omega_2\text{, and }(x_1\notin\Fin(\Omega_1)\text{ or }\Omega\subseteq^\ast x_2)\},
$$
and observe that $\Ss$ is a semifilter. Furthermore, it is clear that $\Ss$ is the union of its complete subspace $\{x\subseteq\omega:x\cap\Omega_1\notin\Fin(\Omega_1)\}$ and its $\sigma$-compact subspace $\{x\subseteq\omega:\Omega\subseteq^\ast x\}$ (see the proof of Proposition \ref{QandQC}). It follows that $\Ss$ is Borel, hence it is strongly homogeneous by Corollary \ref{borelsemifilterstronglyhomogeneous}.

Therefore, in order to apply Theorem \ref{SandT}, it will be enough to show that $\Ss$ is neither $\sigma$-compact nor the union of a complete and a countable subspace. To see that $\Ss$ is not $\sigma$-compact, simply observe that $\PP(\Omega_1)\setminus\Fin(\Omega_1)$ is a closed subspace of $\Ss$ that is homeomorphic to $\PPP$. To see that $\Ss$ is not the union of a complete and a countable subspace, simply observe that $\{x_2\subseteq\Omega_2:\Omega\subseteq^\ast x_2\}$ is a closed subspace of $\Ss$ that is homeomorphic to $\QQQ\times\CCC$.
\end{proof}

\begin{proposition}\label{T}
There exists a semifilter $\TT$ such that $\TT\approx\TTT$.
\end{proposition}
\begin{proof}
Fix infinite sets $\Omega_1$ and $\Omega_2$ such that $\Omega_1\cup\Omega_2=\omega$ and $\Omega_1\cap\Omega_2=\varnothing$. Define
$$
\TT=\{x_1\cup x_2:x_1\subseteq\Omega_1\text{, }x_2\subseteq\Omega_2\text{, and }(x_1\notin\Fin(\Omega_1)\text{ or }x_2\in\Cof(\Omega_2))\},
$$
and observe that $\TT$ is a semifilter. Furthermore, it is clear that $\TT$ is the union of its complete subspace $\{x\subseteq\omega:x\cap\Omega_1\notin\Fin(\Omega_1)\}$ and its countable subspace $\{x_1\cup x_2:x_1\in\Fin(\Omega_1)\text{ and }x_2\in\Cof(\Omega_2)\}$. It follows that $\TT$ is Borel, hence it is strongly homogeneous by Corollary \ref{borelsemifilterstronglyhomogeneous}.

Therefore, in order to apply Theorem \ref{SandT}, it will be enough to show that $\TT$ is neither $\sigma$-compact nor complete. To see that $\TT$ is not $\sigma$-compact, simply observe that $\PP(\Omega_1)\setminus\Fin(\Omega_1)$ is a closed subspace of $\TT$ that is homeomorphic to $\PPP$. To see that $\TT$ is not complete, simply observe that $\Cof(\Omega_2)$ is a closed subspace of $\TT$ that is homeomorphic to $\QQQ$.
\end{proof}

\section{Preliminaries on Wadge classes}

The remaining part of this article relies heavily on results and techniques of van Engelen from \cite{vanengelent} and \cite{vanengeleni}.
In particular, having a copy of \cite{vanengelent} available will be indispensable for the reading of this article.
For this reason, we have tried to follow the notation and terminology of \cite{vanengelent} as closely as possible.

Given a set $Z$ and a collection $\mathbf{\Gamma}\subseteq\PP(Z)$, let $\check{\mathbf{\Gamma}}=\{Z\setminus X:X\in\mathbf{\Gamma}\}$ and $\Delta(\mathbf{\Gamma})=\mathbf{\Gamma}\cap\check{\mathbf{\Gamma}}$ (what $Z$ is will be clear from the context).
The following is \cite[Definition 3.1.1]{vanengelent} (see also \cite[Section 22.E]{kechris}).

\begin{definition}
Let $Z$ be a space, $\eta<\omega_1$ and $1\leq\xi<\omega_1$.
\begin{itemize}
\item Given an increasing sequence of sets $\langle A_\zeta:\zeta<\eta\rangle$, define
$$
\left.
\begin{array}{lcl}
& & \Diff(\langle A_\zeta:\zeta<\eta\rangle)= \left\{
\begin{array}{ll}
\bigcup\{A_\zeta\setminus\bigcup_{\beta<\zeta}A_\beta:\zeta<\eta\text{ and }\zeta\text{ is odd}\} & \text{if }\eta\text{ is even,}\\
\bigcup\{A_\zeta\setminus\bigcup_{\beta<\zeta}A_\beta:\zeta<\eta\text{ and }\zeta\text{ is even}\} & \text{if }\eta\text{ is odd.}
\end{array}
\right.
\end{array}
\right.
$$
\item $X\in \Diff_\eta^Z(\mathbf{\Sigma}^0_\xi)$ if and only if there exists an increasing sequence $\langle A_\zeta:\zeta<\eta\rangle$ of $\mathbf{\Sigma}^0_\xi$ subsets of $Z$ such that $X=\Diff(\langle A_\zeta:\zeta<\eta\rangle)$.
\item $X\in \Diff_\eta(\mathbf{\Sigma}^0_\xi)$ if and only if $Y\in \Diff_\eta^Z(\mathbf{\Sigma}^0_\xi)$ whenever $Z$ is space and $Y$ is a subspace of $Z$ such that $X\approx Y$. The elements of this class are the \emph{absolutely $\Diff_\eta(\mathbf{\Sigma}^0_\xi)$} spaces. Similarly define $\check{\Diff}_\eta(\mathbf{\Sigma}^0_\xi)$.
\item $\Delta=\Diff_\omega(\mathbf{\Sigma}^0_2)\cap\check{\Diff}_\omega(\mathbf{\Sigma}^0_2)$.
\end{itemize}
\end{definition}
For example $\Diff_0(\mathbf{\Sigma}^0_\xi)=\{\varnothing\}$, $\Diff_1(\mathbf{\Sigma}^0_2)$ consists of the $\sigma$-compact spaces, and $\check{\Diff}_1(\mathbf{\Sigma}^0_2)$ consists of the complete spaces. One can easily check that
$$
\Diff_\eta(\mathbf{\Sigma}^0_\xi)\cup\check{\Diff}_\eta(\mathbf{\Sigma}^0_\xi)\subseteq\Diff_\mu(\mathbf{\Sigma}^0_\xi)
$$
whenever $1\leq\xi<\omega_1$ and $\eta<\mu<\omega_1$.

Given a zero-dimensional space $X$, $\eta<\omega_1$, and $2\leq\xi<\omega_1$, using Lavrentiev's Theorem (see \cite[Theorem 3.9]{kechris}), it is not hard to see that $X\in\Diff_\eta(\mathbf{\Sigma}^0_\xi)$ if and only if $X\approx Y$ for some $Y\in \Diff_\eta^{\CCC}(\mathbf{\Sigma}^0_\xi)$.
In particular, a zero-dimensional space $X$ is Borel if and only if $X\approx Y$ for some Borel subspace $Y$ of $\CCC$.
Given a subspace $X$ of $\CCC$, notice that $X\in\Delta$ if and only if $X\in\Delta(\Diff^{\CCC}_\omega(\mathbf{\Sigma}^0_2))$.

\begin{definition}
Assume that $\mathbf{\Gamma}\subseteq\PP(\CCC)$ and $1\leq\xi<\omega_1$. Define $\SU(\mathbf{\Gamma},\mathbf{\Sigma}^0_\xi)$ as the collection of all sets in the form $\bigcup_{n\in\omega}(X_n\cap W_n)$, where the $W_n$ are pairwise disjoint $\mathbf{\Sigma}^0_\xi$ subsets of $\CCC$ and each $X_n\in\mathbf{\Gamma}$.
\end{definition}

Wadge reduction is a fundamental tool in \cite{vanengelent} and \cite{vanengeleni}, as well as in this article. Here, we limit ourselves to the most basic definitions (see \cite[Sections 4.1 and 4.2]{vanengelent} for a more comprehensive treatment).
Given $A,B\subseteq\CCC$, we will write $A\leq_W B$ to mean that there exists a continuous function $f:\CCC\longrightarrow\CCC$
such that $A=f^{-1}[B]$. In this case, we will say that \emph{$A$ is Wadge-reducible to $B$}. Given $X\subseteq\CCC$, let $[X]=\{A\subseteq\CCC: A\leq_W X\}$.
We will say that $\mathbf{\Gamma}$ is a \emph{Wadge class} if there exists $X\subseteq\CCC$, such that $\mathbf{\Gamma}=[X]$. In this case, we will say that \emph{$X$ generates $\mathbf{\Gamma}$}.
A \emph{Borel Wadge class} is a Wadge class that consists only of Borel sets. A Wadge class is \emph{self-dual} if $\check{\mathbf{\Gamma}}=\mathbf{\Gamma}$.

The following is the fundamental result in the theory of Borel Wadge classes (see \cite[Theorem 21.14]{kechris}).
\begin{lemma}[Wadge]\label{wadge}
Let $X$ and $Y$ be Borel subsets of $\CCC$. Then either $X\leq_W Y$ or $Y\leq_W \CCC\setminus Y$.
\end{lemma}

Following the results of Louveau from \cite{louveau}, van Engelen defines a certain subset $D$ of $\omega_1^\omega$,
then he associates a Wadge class $\mathbf{\Gamma}_u$ to each $u\in D$ so that the following fundamental theorem holds.
For every $u\in D$, he also defines the \emph{type of $u$} as a suitable $t(u)\in\{0,1,2,3\}$ (see \cite[Definition 4.2.3]{vanengelent}).
We do not think it would be particularly useful or enlightening to give all the details here.
We will only mention that if $u=\xi^\frown 1^\frown\eta^\frown\underline{0}$, where $1\leq\xi,\eta<\omega_1$, then $u\in D$ and $\mathbf{\Gamma}_u=\Diff^{\CCC}_\eta(\mathbf{\Sigma}^0_\xi)$. Here $\underline{0}$ denotes the element of $\omega_1^\omega$ which is constant with value $0$.

\begin{theorem}[Louveau for $\PPP$; van Engelen for $\CCC$]\label{louveautheorem}
The collection of non-self-dual Borel Wadge classes is precisely
$$
\{\mathbf{\Gamma}_u:u\in D\}\cup\{\check{\mathbf{\Gamma}}_u:u\in D\}.
$$
\end{theorem}
\begin{proof}
By \cite[Theorem 4.2.7]{vanengelent}, it will be enough to show that $\mathbf{\Gamma}_u$ is non-self-dual for each $u\in D$,
and this can be done as in the proof of \cite[Proposition 1.3]{louveau}. The desired result also follows from items 4 and 1 on \cite[page 90]{louveausaintraymond}.
\end{proof}

\section{The classification of homogeneous spaces: the case below $\Delta$}

This section is a minimalist introduction to the classification of the homogeneous zero-dimensional (Borel) spaces that are in $\Delta$ and not locally compact.

The following is essentially \cite[Definition 3.1.7]{vanengelent} (see also \cite[Lemma 3.1.4]{vanengelent}). It is included for completeness, since we will only need the obvious fact that if $X\in \Diff_\ell(\mathbf{\Sigma}^0_2)$ for some $\ell\in\omega$ then $X$ has one of the following properties.

\begin{definition}[van Engelen]
Let $X$ be a space, and let $k\in\omega$.
\begin{itemize}
\item $X$ is $\PP_{4k}$ if and only if $X$ is the union of a subspace in $\Diff_{2k}(\mathbf{\Sigma}^0_2)$ and a complete subspace.
\item $X$ is $\PP_{4k+1}$ if and only if $X\in \Diff_{2(k+1)}(\mathbf{\Sigma}^0_2)$.
\item $X$ is $\PP^1_{4k+2}$ if and only if $X$ is the union of a subspace in $\Diff_{2k}(\mathbf{\Sigma}^0_2)$, a complete subspace, and a countable subspace.
\item $X$ is $\PP^1_{4k+3}$ if and only if $X$ is the union of a subspace in $\Diff_{2(k+1)}(\mathbf{\Sigma}^0_2)$ and a countable subspace.
\item $X$ is $\PP^2_{4k+2}$ if and only if $X$ is the union of a subspace in $\Diff_{2k}(\mathbf{\Sigma}^0_2)$, a complete subspace, and a $\sigma$-compact subspace.
\item $X$ is $\PP^2_{4k+3}$ if and only if $X$ is the union of a subspace in $\Diff_{2(k+1)}(\mathbf{\Sigma}^0_2)$ and a $\sigma$-compact subspace.
\end{itemize}
Also define $\PP^1_{-1}$ as the property of being countable, and $\PP^2_{-1}$ as the property of being $\sigma$-compact. To indicate one of these properties generically (that is, in case we do not know whether the superscript $i\in\{1,2\}$ is present or not) we will use the notation $\PP_n^{(i)}$.
\end{definition}

In \cite[Theorem 3.4.24]{vanengelent}, the following order is declared on these properties.
$$
\PP_{-1}^1<\PP_{-1}^2<\cdots<\PP_{4k}<\PP_{4k+1}<\PP^1_{4k+2}<\PP^1_{4k+3}<\PP^2_{4k+2}<\PP^2_{4k+3}<\cdots
$$
In \cite[Definition 3.4.6]{vanengelent}, van Engelen defines the classes $\XX^1_{-1}$, $\XX^2_{-1}$, and $\XX_{4k}$, $\XX_{4k+1}$, $\XX^1_{4k+2}$, $\XX^1_{4k+3}$, $\XX^2_{4k+2}$, $\XX^2_{4k+3}$ for $k\in\omega$.
As with the properties $\PP_n^{(i)}$, we will use the notation $\XX_n^{(i)}$ to indicate one of these classes generically.

The following two results are the most important facts about the classes $\XX_n^{(i)}$ (see \cite[Theorem 3.4.13]{vanengelent} and \cite[Theorem 3.4.24]{vanengelent}). 
In fact, we will not give the general definition of $\XX_n^{(i)}$, but use the more easily understandable Theorem \ref{belowdeltabelongs} instead.
Notice that, by Theorem \ref{belowdeltabelongs}, the class $\XX_n^{(i)}$ is closed under homeomorphisms for every $n\in\omega$ and $i\in\{1,2\}$.

\begin{theorem}[van Engelen]\label{belowdelta1}
Let $n\in\omega$ and $i\in\{1,2\}$. Then, up to homeomorphism, the class $\XX_n^{(i)}$ contains exactly one element, which is homogeneous.
\end{theorem}
\begin{theorem}[van Engelen]\label{belowdeltabelongs}
Let $n\in\omega$ and $i\in\{1,2\}$. Then, for a zero-dimensional space $X$, the following conditions are equivalent.
\begin{itemize}
\item $X\in\XX_n^{(i)}$.
\item $X$ is $\PP_n^{(i)}$ and nowhere $\PP_m^{(j)}$ for every $m\in\{-1\}\cup\omega$ and $j\in\{1,2\}$ such that $\PP_m^{(j)}<\PP_n^{(i)}$.
\end{itemize}
\end{theorem}
To complete the picture, let us define the classes that are not covered by Theorem \ref{belowdeltabelongs}.
In \cite[Definition 3.4.6]{vanengelent}, these classes are defined as singletons, but it seems clear that they should be closed under homeomorphism.
\begin{itemize}
\item $\XX^1_{-1}$ is the class of spaces that are homeomorphic to $\QQQ$.
\item $\XX^2_{-1}$ is the class of spaces that are homeomorphic to $\QQQ\times\CCC$.
\end{itemize}
Similarly, $\XX_0$ should be defined as the class of spaces that are homeomorphic to $\PPP$.

The following diagram (which is taken from \cite[page 28]{vanengelent}) illustrates the first few classes $\XX_n^{(i)}$. Spaces that appear at the same level generate the same class when embedded into $\CCC$ (see \cite[Theorems 4.6.4 and 4.6.5]{vanengelent}). In particular, Wadge class and Baire category are not sufficient to determine the homemorphism type. This phenomenon starts with $\QQQ$ and $\QQQ\times\CCC$ and propagates throughout $\Delta$. As we will see in Section 9, this ambiguity disappears for spaces that are not in $\Delta$.

\newpage

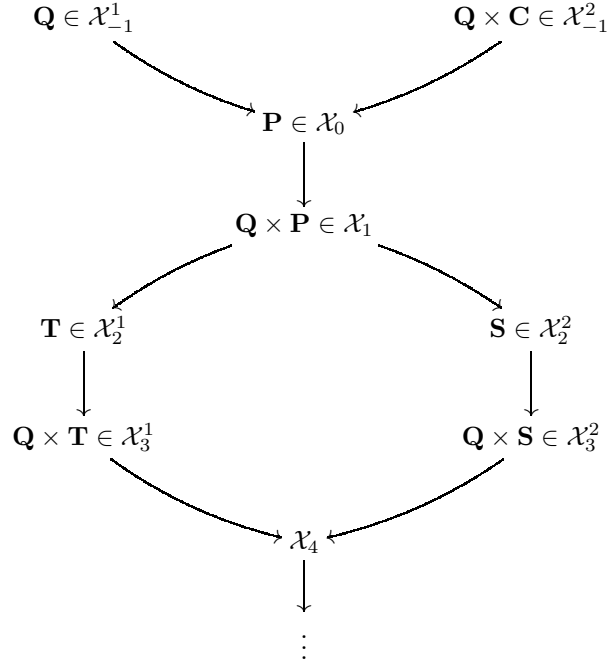
\begin{center}
$
\xymatrix{
\QQQ\in\XX_{-1}^1 \ar@/_/[rd] & & \QQQ\times\CCC\in\XX_{-1}^2 \ar@/^/[ld]\\
& \PPP\in\XX_0 \ar[d] &\\
& \QQQ\times\PPP\in\XX_1 \ar@/_/[ld] \ar@/^/[rd] &\\
\TTT\in\XX_2^1 \ar[d] & & \SSS\in\XX_2^2 \ar[d]\\
\QQQ\times\TTT\in\XX_3^1 \ar@/_/[rd] & & \QQQ\times\SSS\in\XX_3^2 \ar@/^/[ld]\\
& \XX_4 \ar[d] &\\
& \vdots &\\
}
$
\end{center}

\bigskip

The following theorem is one of the reasons why the class $\Delta$ plays such an important role (the other reason is Lemma \ref{reasonably}).
It does not appear explicitly in \cite{vanengelent}, but Lemma \ref{indeltaimpliesinDn} is mentioned (without giving any precise reference) in the proofs of \cite[Lemmas 2.4 and 2.5]{vanengeleni}.
The rest of this section is devoted to its proof.

\begin{theorem}\label{fromdeltatoX}
Let $X$ be a zero-dimensional homogeneous space such that $X\in\Delta$ and $X$ is not locally compact. Then $X\in\XX^{(i)}_n$ for some $n\in\omega\cup\{-1\}$ and $i\in\{1,2\}$.
\end{theorem}
\begin{proof}
Simply embed $X$ in $\CCC$, then apply Lemmas \ref{indeltaimpliesinDn} and \ref{inDnimpliesinXk}.
\end{proof}

For the proof following lemma, we will need a couple more definitions. As in \cite[Definition 3.5.7]{vanengelent}, define $\PP_\omega$ as the property of being $\Diff_\omega(\mathbf{\Sigma}^0_2)$. As in \cite[Definition 3.1.8]{vanengelent} (see also the remark at the beginning of \cite[Section 3.5]{vanengelent}), define $\XX^2_\omega$ as the class of all spaces that are $\PP_\omega$ and nowhere $\PP_n^{(i)}$ for every $n\in\omega$ and $i\in\{1,2\}$. See Section 9 for the definition of $\YY_u^0$.

\begin{lemma}\label{indeltaimpliesinDn}
Let $X$ be a homogeneous subspace of $\CCC$. Assume that $X\in\Delta$. Then $X\in \Diff_\ell(\mathbf{\Sigma}^0_2)$ for some $\ell\in\omega$.
\end{lemma}
\begin{proof}
Since $X\in\Delta\subseteq \Diff_\omega(\mathbf{\Sigma}^0_2)$, we can fix the least $\alpha\in\omega\cup\{\omega\}$ such that $X\in\Diff_\alpha(\mathbf{\Sigma}^0_2)$. Assume, in order to get a contradiction, that $\alpha=\omega$. It follows from \cite[Lemma 3.6.1]{vanengelent} that $X$ is $\PP_\omega$ and nowhere $\PP_n^{(i)}$ for every $n\in\omega$ and $i\in\{1,2\}$. This means that $X\in\XX_\omega^2$. However $\XX_\omega^2=\YY_u^0$ by \cite[Theorem 4.6.2(a)]{vanengelent}, where $u=2^\frown 1^\frown\omega^\frown \underline{0}$.
This contradicts the definition of $\YY_u^0$, because $X\in\Delta(\Diff^{\CCC}_\omega(\mathbf{\Sigma}^0_2))\subseteq\check{\Diff}^{\CCC}_\omega(\mathbf{\Sigma}^0_2)=\check{\mathbf{\Gamma}}_u$.
\end{proof}

\newpage

The proof of the following lemma is taken from the proof of \cite[Theorem 3.6.2]{vanengelent}.

\begin{lemma}\label{inDnimpliesinXk}
Let $X$ be a zero-dimensional homogeneous space. Assume that $X$ is not locally compact and $X\in \Diff_\ell(\mathbf{\Sigma}^0_2)$ for some $\ell\in\omega$. Then $X\in\XX^{(i)}_n$ for some $n\in\omega\cup\{-1\}$ and $i\in\{1,2\}$.
\end{lemma}
\begin{proof}
Let $<$ denote the linear ordering on the properties $\PP^{(j)}_m$ for $m\in\omega\cup\{-1\}$ and $j\in\{1,2\}$ defined earlier in this section. First notice that $X$ is $\PP^{(i)}_n$ for some $n\in\omega$ and $i\in\{1,2\}$. For example, picking any $k\in\omega$ such that $2(k+1)\geq\ell$ will show that $X$ is $\PP_{4k+1}$.
Therefore, we can fix $n\in\omega$ and $i\in\{1,2\}$ such that $\PP^{(i)}_n$ is the minimal property with respect to $<$ such that $X$ is $\PP^{(i)}_n$.

Notice that $X$ is nowhere compact because it is homogeneous and not locally compact.
Therefore, if $X$ is $\sigma$-compact then either $X\approx\QQQ$ (if $X$ is countable) or $X\approx\QQQ\times \CCC$ (if $X$ is nowhere countable, by \cite[Theorem 2.4.5]{vanengelent}).
Notice that $X\in\XX^1_{-1}$ in the first case and $X\in\XX^2_{-1}$ in the second case. So assume that $X$ is not $\sigma$-compact.

By \cite[Lemma 3.6.1]{vanengelent}, it follows that $X$ is nowhere $\PP^{(j)}_m$ whenever $m\in\omega\cup\{-1\}$ and $j\in\{1,2\}$
are such that $\PP^{(j)}_m<\PP^{(i)}_n$. Therefore $\XX\in\XX^{(i)}_n$ by Theorem \ref{belowdeltabelongs}.
\end{proof}

\section{From homogeneous space to semifilter: the case below $\Delta$}

In the proof of Theorem \ref{indelta}, we will need the fact that certain operations, when applied to a semifilter, yield a space that is still homeomorphic to a semifilter.
The following three lemmas make this explicit. Notice that Lemma \ref{semifiltercomplement} cannot be generalized to filters.

\begin{lemma}\label{semifiltercomplement}
Let $\Ss$ be a semifilter. Then $\CCC\setminus\Ss$ is homeomorphic to a semifilter.
\end{lemma}
\begin{proof}
It is straightforward to check that $\CCC\setminus\Ss$ is a semiideal. It follows that $\TT=c[\CCC\setminus\Ss]$ is a semifilter.
The trivial fact that $\TT\approx\CCC\setminus\Ss$ concludes the proof.
\end{proof}

\begin{lemma}\label{semifilterrational}
Let $\Ss$ be a semifilter. Then $\QQQ\times\Ss$ is homeomorphic to a semifilter.
\end{lemma}
\begin{proof}
Let $\Omega_1$ and $\Omega_2$ be infinite sets such that $\Omega_1\cup\Omega_2=\omega$ and $\Omega_1\cap\Omega_2=\varnothing$.
Fix a bijection $\pi:\omega\longrightarrow\Omega_2$, and let $\Ss(\Omega_2)=\{\pi[x]:x\in\Ss\}$.
It is clear that $\Ss(\Omega_2)$ is a semifilter on $\Omega_2$ that is homeomorphic to $\Ss$. Define
$$
\TT=\{x_1\cup x_2:x_1\in\Cof(\Omega_1)\text{ and }x_2\in\Ss(\Omega_2)\},
$$
and observe that $\TT$ is a semifilter. The trivial fact that $\TT\approx\Cof(\Omega_1)\times\Ss(\Omega_2)$ concludes the proof.
\end{proof}

\begin{lemma}\label{semifiltertricky}
Let $\Ss$ be a semifilter. Then $(\CCC\times\CCC)\setminus (\Fin\times (\CCC\setminus\Ss))$ is homeomorphic to a semifilter.
\end{lemma}
\begin{proof}
Let $\Omega_1$ and $\Omega_2$ be infinite sets such that $\Omega_1\cup\Omega_2=\omega$ and $\Omega_1\cap\Omega_2=\varnothing$. Fix bijections $\sigma:\omega\longrightarrow\Omega_1$ and $\pi:\omega\longrightarrow\Omega_2$.
Let $\Ss(\Omega_2)=\{\pi[x]:x\in\Ss\}$, and define
$$
\TT=\{x_1\cup x_2:x_1\subseteq\Omega_1\text{, }x_2\subseteq\Omega_2\text{, and }(x_1\notin\Fin(\Omega_1)\text{ or }x_2\in\Ss(\Omega_2))\},
$$
and observe that $\TT$ is a semifilter. We will show that $\TT$ is the desired semifilter.

It is clear that $\sigma$ and $\pi$ induce homeomorphisms $h:\CCC\longrightarrow 2^{\Omega_1}$ and $g:\CCC\longrightarrow2^{\Omega_2}$ respectively
such that $h[\Fin(\Omega_1)]=\Fin$ and $g[\Ss(\Omega_2)]=\Ss$. It follows that $h\times g:\CCC\times\CCC\longrightarrow 2^{\Omega_1}\times 2^{\Omega_2}$
is a homeomorphism such that
$$
(h\times g)[(\CCC\times\CCC)\setminus (\Fin\times (\CCC\setminus\Ss))]=(2^{\Omega_1}\times 2^{\Omega_2})\setminus (\Fin(\Omega_1)\times (2^{\Omega_2}\setminus\Ss(\Omega_2))).
$$
The trivial fact that $\TT\approx (2^{\Omega_1}\times 2^{\Omega_2})\setminus (\Fin(\Omega_1)\times (2^{\Omega_2}\setminus\Ss(\Omega_2)))$ concludes the proof.
\end{proof}

\begin{theorem}\label{indelta}
Let $X$ be a zero-dimensional homogeneous space. Assume that $X\in\Delta$ and $X$ is not locally compact. Then there exists a semifilter $\Ss$ such that $\Ss\approx X$.
\end{theorem}
\begin{proof}
We will use induction to show that $\XX^{(i)}_n$ contains a semifilter for every $n\in\omega\cup\{-1\}$ and $i\in\{1,2\}$. By Theorems \ref{fromdeltatoX} and \ref{belowdelta1}, this will be enough.

The case $n=-1$ is the base of our induction, and it follows from Proposition \ref{QandQC}. For the inductive step, assume that the claim holds for $\XX^1_{4k-1}$ and $\XX^2_{4k-1}$, where $k\in\omega$ is given. We will show that the claim holds for $\XX_{4k}$, $\XX_{4k+1}$, $\XX^1_{4k+2}$, $\XX^2_{4k+2}$, $\XX^1_{4k+3}$ and $\XX^2_{4k+3}$.

\textbf{Case 1}: $\XX_{4k}$. Fix $i\in\{1,2\}$ (either one will work). By the inductive assumption, there exists a semifilter $\Ss\in\XX^i_{4k-1}$. By \cite[Lemma 3.4.11(b)]{vanengelent}, it follows that $\CCC\setminus\Ss\in\XX_{4k}$.
An application of Lemma \ref{semifiltercomplement} concludes the proof in this case.

\textbf{Case 2}: $\XX_{4k+1}$. As we just showed, there exists a semifilter $\Ss\in\XX_{4k}$. By \cite[Lemma 3.4.9]{vanengelent}, it follows that $\QQQ\times\Ss\in\XX_{4k+1}$.
An application of Lemma \ref{semifilterrational} concludes the proof in this case.

\textbf{Case 3}: $\XX^i_{4k+2}$. Fix $i\in\{1,2\}$. By the inductive hypothesis, there exists a semifilter $\Ss_i\in\XX^i_{4k-1}$. Case 2 of the proof of \cite[Lemma 3.4.12]{vanengelent} shows that $Y_i\in\XX^i_{4k+3}$ whenever $Q$ is a countable dense subspace of $\CCC$ and $X_i\in\XX^i_{4k-1}$ is a dense subspace of $\CCC$, where
$$
Y_i=(\CCC\times \CCC)\setminus (Q\times (\CCC\setminus X_i)).
$$
Therefore, it will be enough to set $Q=\Fin$, $X_i=\Ss_i$, and apply Lemma \ref{semifiltertricky}.

\textbf{Case 4}: $\XX^i_{4k+3}$. Fix $i\in\{1,2\}$. As we just showed, there exists a semifilter $\Ss_i\in\XX^i_{4k+2}$.
By \cite[Lemma 3.4.9]{vanengelent}, it follows that $\QQQ\times\Ss_i\in\XX^i_{4k+3}$.
An application of Lemma \ref{semifilterrational} concludes the proof.
\end{proof}

\section{The classification of homogeneous spaces: the case above $\Delta$}

This section is a minimalist introduction to the classification of the homogeneous zero-dimensional Borel spaces that are not in $\Delta$.
The following are \cite[Definitions 4.3.1 and 4.3.2]{vanengelent}.

\begin{definition}[van Engelen]\label{Pu}
Let $u\in D$. Assume that $X$ is a zero-dimensional space.
\begin{itemize}
\item $X$ is $\PP_u$ if and only if $Y\approx X$ implies $Y\in\mathbf{\Gamma}_u$ for every subspace $Y$ of $\CCC$.
\item $X$ is $\check{\PP}_u$ if and only if $Y\approx X$ implies $Y\in\check{\mathbf{\Gamma}}_u$ for every subspace $Y$ of $\CCC$.
\end{itemize}
\end{definition}

\begin{definition}[van Engelen]\label{YuZu}
Let $u\in D$. Assume that $X$ is a zero-dimensional space.
\begin{itemize}
\item $X\in\YY_u^0$ if and only if $X$ is $\PP_u$, nowhere $\check{\PP}_u$, and first category.
\item $X\in\YY_u^1$ if and only if $X$ is $\PP_u$, nowhere $\check{\PP}_u$, and Baire.
\item $X\in\ZZ_u^0$ if and only if $X$ is $\check{\PP}_u$, nowhere $\PP_u$, and first category.
\item $X\in\ZZ_u^1$ if and only if $X$ is $\check{\PP}_u$, nowhere $\PP_u$, and Baire.
\end{itemize}
\end{definition}

\newpage

The following two results classify the homogeneous zero-dimensional Borel spaces not in $\Delta$.
The first one follows from \cite[Lemma 4.3.5, Lemma 4.3.6, and Theorem 4.3.8]{vanengelent}, and the second one is \cite[Theorem 4.4.4]{vanengelent}.
Recall the following definitions from \cite[page 100]{vanengelent}.
\begin{itemize}
\item $D_0=\{u\in D:\Delta(\Diff^{\CCC}_\omega(\mathbf{\Sigma}^0_2))\subseteq\mathbf{\Gamma}_u\text{ and }u(0)\geq 2\}$.
\item $D_1=\{u\in D_0:u(0)\geq 3\text{ or }t(u)=3\}$.
\end{itemize}

\begin{theorem}[van Engelen]\label{abovedelta1}
If $u\in D_0$ then, up to homeomorphism, both $\YY^0_u$ and $\ZZ^1_u$ contain exactly one element, which is homogeneous.
If $u\in D_1$ then, up to homeomorphism, both $\YY^1_u$ and $\ZZ^0_u$ contain exactly one element, which is homogeneous. 
\end{theorem}

\begin{theorem}[van Engelen]\label{abovedeltabelongs}
Let $X$ be a homogeneous zero-dimensional Borel space such that $X\notin\Delta$. Then $X\in\YY^0_u\cup\ZZ^1_u$ for some $u\in D_0$ or $X\in\YY^1_u\cup\ZZ^0_u$ for some $u\in D_1$.
\end{theorem}

Finally, the following proposition makes explicit which Wadge classes are generated by the spaces considered above.

\begin{proposition}\label{belongsgenerate}
Let $X$ be a subspace of $\CCC$. If $X\in\YY^0_u\cup\YY^1_u$ for some $u\in D_0$ then $[X]=\mathbf{\Gamma}_u$. If $X\in\ZZ^0_u\cup\ZZ^1_u$ for some $u\in D_0$ then $[X]=\check{\mathbf{\Gamma}}_u$.
\end{proposition}
\begin{proof}
Since the other case is similar, we will assume that $X\in\YY^0_u\cup\YY^1_u$ for some $u\in D_0$. This means in particular that $X\in\mathbf{\Gamma}_u$, hence $[X]\subseteq\mathbf{\Gamma}_u$. Since $X$ is not $\check{\PP}_u$, there must be a subspace $Y$ of $\CCC$ such that $Y\approx X$ and $Y\notin\check{\mathbf{\Gamma}}_u$. By \cite[Lemma 4.2.16]{vanengelent}, it follows that $X\notin\check{\mathbf{\Gamma}}_u$. Using Lemma \ref{wadge}, one sees that this implies $\mathbf{\Gamma}_u\subseteq [X]$.
\end{proof}

\section{Some useful results}

In this section, we collect some miscellaneous results that will be useful later. Each one of them was either explicitly stated by van Engelen, or follows easily from his results.

\begin{lemma}\label{nonselfdual}
Assume that $X$ is a homogeneous Borel subspace of $\CCC$ that is not locally compact. Let $\mathbf{\Gamma}=[X]$.
\begin{itemize}
\item Then $\mathbf{\Gamma}\in\{\mathbf{\Gamma}_u,\check{\mathbf{\Gamma}}_u\}$ for some $u\in D$ with $u(0)\geq 2$.
\item If $X\notin\Delta$ then $\mathbf{\Gamma}\in\{\mathbf{\Gamma}_u,\check{\mathbf{\Gamma}}_u\}$ for some $u\in D_0$.
\end{itemize}
In particular, $\mathbf{\Gamma}$ is non-self-dual.
\end{lemma}
\begin{proof}
If $X\notin\Delta$, the desired result follows immediately from Theorem \ref{abovedeltabelongs} and Proposition \ref{belongsgenerate}. Now assume that $X\in\Delta$. Notice that, by Theorem \ref{fromdeltatoX}, we can fix $n\in\{-1\}\cup\omega$ and $i\in\{1,2\}$ such that $X\in\XX_n^{(i)}$. If $n=-1$, then $X$ is a $\mathbf{\Sigma}^0_2$ subspace of $\CCC$ that is not $\mathbf{\Pi}^0_2$, hence $[X]=\Diff_1^{\CCC}(\mathbf{\Sigma}^0_2)=\mathbf{\Gamma}_u$ by Lemma \ref{wadge}, where $u=2^\frown 1^\frown 1^\frown\underline{0}$. If $n=0$, then $X$ is a $\mathbf{\Pi}^0_2$ subspace of $\CCC$ that is not $\mathbf{\Sigma}^0_2$, hence $[X]=\check{\Diff}_1^{\CCC}(\mathbf{\Sigma}^0_2)=\check{\mathbf{\Gamma}}_u$ by Lemma \ref{wadge}, where $u=2^\frown 1^\frown 1^\frown\underline{0}$. If $n>0$, then the desired result follows from \cite[Lemma 4.6.4]{vanengelent}. The fact that $\mathbf{\Gamma}$ is non-self-dual follows from Theorem \ref{louveautheorem}.
\end{proof}

The following is \cite[Lemma 2.4(c)]{vanengeleni}. It shows that, no matter how a zero-dimensional homogeneous Borel space is embedded in $\CCC$, the Wadge class that it generates will remain the same (with the trivial exception of locally compact spaces).

\begin{lemma}[van Engelen]\label{wadgeindependent}
Let $X$ be a homogeneous Borel subspace of $\CCC$ that is not locally compact. Assume that $Y$ is a subspace of $\CCC$ such that $Y\approx X$. Then $X\equiv_W Y$.
\end{lemma}
\begin{proof}
If $X\notin\Delta$, the desired result follows from Theorem \ref{abovedeltabelongs} and Proposition \ref{belongsgenerate}. If $X\in\Delta$, simply observe that $[X]=[Y]$ by the second part of the proof of Lemma \ref{nonselfdual}.
\end{proof}

The next three results clarify the closure properties of certain Wadge classes. The second part of Lemma \ref{separatedunions} is a particular case of \cite[Lemma 2.4(b)]{vanengeleni}. Lemma \ref{deltaxi} follows from \cite[Lemma 4.2.11(b)]{vanengelent}. Lemma \ref{GdeltaFsigma} is \cite[Lemma 4.2.12]{vanengelent}.

\begin{lemma}[van Engelen]\label{separatedunions}
Let $X$ be a homogeneous Borel subspace of $\CCC$ that is not locally compact, and let $\mathbf{\Gamma}=[X]$.
\begin{itemize}
\item Then $\SU(\mathbf{\Sigma}^0_1,\mathbf{\Gamma})=\mathbf{\Gamma}$.
\item If $X$ is first category and $X\notin\Delta$ then $\SU(\mathbf{\Sigma}^0_2,\mathbf{\Gamma})=\mathbf{\Gamma}$.
\end{itemize}
\end{lemma}
\begin{proof} 
By Lemma \ref{nonselfdual}, we have $\mathbf{\Gamma}\in\{\mathbf{\Gamma}_u,\check{\mathbf{\Gamma}}_u\}$ for some $u\in D$ with $u(0)\geq 2$. Therefore, it follows from \cite[Lemma 4.2.11(a)]{vanengelent} that $\SU(\mathbf{\Sigma}^0_1,\mathbf{\Gamma})\subseteq \SU(\mathbf{\Sigma}^0_\xi,\mathbf{\Gamma})=\mathbf{\Gamma}$, where $\xi=u(0)$. The trivial fact that $\mathbf{\Gamma}\subseteq \SU(\mathbf{\Sigma}^0_1,\mathbf{\Gamma})$ concludes the proof of the first statement.

Now assume that $X$ is first category and $X\notin\Delta$. By Theorem \ref{abovedeltabelongs}, either there exists $u\in D_0$ such that $X\in\YY_u^0\cup\ZZ_u^1$ or there exists $u\in D_1$ such that $X\in\YY_u^1\cup\ZZ_u^0$. By Definition \ref{YuZu}, the cases $X\in\ZZ_u^1$ and $X\in\YY_u^1$ are impossible. First assume that $X\in\YY_u^0$ for some $u\in D_0$. Observe that $\mathbf{\Gamma}=\mathbf{\Gamma}_u$ by Proposition \ref{belongsgenerate}. It follows from \cite[Lemma 4.2.11(a)]{vanengelent} that $\SU(\mathbf{\Sigma}^0_2,\mathbf{\Gamma})=\mathbf{\Gamma}$. Finally, assume that $X\in\ZZ_u^0$ for some $u\in D_1$. Observe that $\mathbf{\Gamma}=\check{\mathbf{\Gamma}}_u$ by Proposition \ref{belongsgenerate}. Furthermore, by the definition of $D_1$, either $u(0)\geq 3$ or $u(0)=2$ and $t(u)=3$. By \cite[Corollary 4.2.14]{vanengelent}, it follows that $\SU(\mathbf{\Sigma}^0_2,\mathbf{\Gamma})=\mathbf{\Gamma}$.
\end{proof}

\begin{lemma}[van Engelen]\label{deltaxi}
Assume that $\mathbf{\Gamma}\in\{\mathbf{\Gamma}_u,\check{\mathbf{\Gamma}}_u\}$ for some $u\in D$ such that $u(0)\geq 2$. If $X\in\mathbf{\Gamma}$ and $H\in\mathbf{\Delta}^0_2$ then $X\cap H\in\mathbf{\Gamma}$.
\end{lemma}

\begin{lemma}[van Engelen]\label{GdeltaFsigma}
Assume that $\mathbf{\Gamma}\in\{\mathbf{\Gamma}_u,\check{\mathbf{\Gamma}}_u\}$ for some $u\in D_0$.
\begin{itemize}
\item If $X\in\mathbf{\Gamma}$ and $G$ is a $\mathbf{\Pi}^0_2$ subset of $\CCC$ then $X\cap G\in\mathbf{\Gamma}$.
\item If $X\in\mathbf{\Gamma}$ and $F$ is a $\mathbf{\Sigma}^0_2$ subset of $\CCC$ then $X\cup F\in\mathbf{\Gamma}$.
\end{itemize}
\end{lemma}

The following result is essentially \cite[Lemma 2.11]{vanengeleni}.

\begin{lemma}[van Engelen]\label{closedwadge}
Let $X$ and $Y$ be Borel subspaces of $\CCC$ such that $X$ is a closed subset of $Y$. Let $\mathbf{\Gamma}=[X]$, and assume that $\mathbf{\Gamma}\in\{\mathbf{\Gamma}_u,\check{\mathbf{\Gamma}}_u\}$ for some $u\in D$ such that $u(0)\geq 2$. Then $X\leq_W Y$.
\end{lemma}
\begin{proof}
In order to get a contradiction, suppose that $X\not\leq_W Y$. Then $Y\leq_W \CCC\setminus X$ by Lemma \ref{wadge}, hence $Y\in\check{\mathbf{\Gamma}}$. Since $X=Y\cap H$ for some closed subset $H$ of $\CCC$, it follows from Lemma \ref{deltaxi} that $X\in\check{\mathbf{\Gamma}}$, which contradicts the fact that $\mathbf{\Gamma}$ is non-self-dual.
\end{proof}

The following theorem shows that, no matter how one embeds a homogeneous zero-dimensional Borel space $X$ into $\CCC$, the embedding will be extremely nice, provided it is a dense embedding and $X\notin\Delta$. It is inspired by the notion of being h-homogeneously embedded in a space, which first appeared in \cite[Definition 7.1]{medinic}. Notice that, by \cite[Theorem 6]{medinivanmillzdomskyy}, the assumption that $X$ is Borel cannot be dropped in Corollary \ref{homogeneouscomplement}.

\begin{theorem}\label{hhomogeneouslyembedded}
Let $X$ be a homogeneous Borel dense subspace of $\CCC$. Assume that $X\notin\Delta$. Then, for every non-empty clopen subset $U$ of $\CCC$, there exists a homeomorphism $h_U:\CCC\longrightarrow U$ such that $h_U[X]=U\cap X$.
\end{theorem}
\begin{proof}
Since $X\notin\Delta$, it follows from Theorem \ref{abovedeltabelongs} that $X\in\YY^0_u\cup\YY^1_u\cup\ZZ^0_u\cup\ZZ^1_u$ for some $u\in D_0$. Therefore, by \cite[Theorem 4.3.9]{vanengelent}, if $Y$ is a dense subspace of $\CCC$ such that $Y\approx X$, then there exists a homeomorphism $h:\CCC\longrightarrow\CCC$ such that $h[X]=Y$. Furthermore, since $X$ is not locally compact, it follows from \cite[Corollary 4.4.6]{vanengelent} that $X$ is strongly homogeneous.

Now let $U$ be a non-empty clopen subset of $\CCC$. Notice that $U\approx\CCC$ and $U\cap X$ is a dense subspace of $U$. Furthermore, the fact that $X$ is strongly homogeneous
implies that $U\cap X\approx X$. Hence, the existence of the desired homeomorphism follows from the observations in the first part of this proof.
\end{proof}

\begin{corollary}\label{homogeneouscomplement}
Let $X$ be a homogeneous Borel dense subspace of $\CCC$. Assume that $X\notin\Delta$. Then $\CCC\setminus X$ is homogeneous.
\end{corollary}
\begin{proof}
Using Theorem \ref{hhomogeneouslyembedded}, it is easy to see that $\CCC\setminus X$ has a base consisting of clopen subspaces that are homeomorphic to $\CCC\setminus X$. By Lemma \ref{pibase}, it follows that $\CCC\setminus X$ is strongly homogeneous, hence homogeneous.
\end{proof}

\section{A theorem of Steel}

The following is a particular case of \cite[Theorem 2]{steel}. It is the fundamental tool for dealing with homogeneous Borel spaces that are not in $\Delta$ (here, as well as in \cite{vanengelent} and \cite{vanengeleni}). We remark that the proof of Theorem \ref{steeltheorem} uses the fact that all Borel games are determined (this is a deep result due to Martin, see \cite[Theorem 20.5]{kechris}). As usual, we will follow closely the exposition of van Engelen (see \cite[Section 4.1]{vanengelent}).

Let $\mathbf{\Gamma}\subseteq\PP(\CCC)$ and $X\subseteq \CCC$ be given. We will say that $X$ is \emph{everywhere properly $\mathbf{\Gamma}$} if $U\cap X\in\mathbf{\Gamma}\setminus\check{\mathbf{\Gamma}}$ for every non-empty open subset $U$ of $\CCC$. We will say that $\mathbf{\Gamma}$ is \emph{continuously closed} if $f^{-1}[A]\in\mathbf{\Gamma}$ whenever $A\in\mathbf{\Gamma}$ and $f:\CCC\longrightarrow\CCC$ is a continuous function.

Given $i\in\{0,1\}$, define $Q_i=\{x\in \CCC:x(n)=i\text{ for all but finitely many }n\in\omega\}$. Also define the function $\phi:\CCC\setminus (Q_0\cup Q_1)\longrightarrow \CCC$ by setting
$$
\left.
\begin{array}{lcl}
& & \phi(x)(n)= \left\{
\begin{array}{ll}
0 & \text{if the } n\text{-th block of zeros in }x\text{ has even length,}\\
1 & \text{if the } n\text{-th block of zeros in }x\text{ has odd length,}
\end{array}
\right.
\end{array}
\right.
$$
where we start counting from the $0$-th block. Given $\mathbf{\Gamma}\subseteq\PP(\CCC)$, we will say that $\mathbf{\Gamma}$ is \emph{reasonably closed} if it is continuously closed and $\phi^{-1}[A]\cup Q_0\in\mathbf{\Gamma}$ whenever $A\in\mathbf{\Gamma}$.
The only fact involving reasonably closed classes that we will need is the following lemma, as it will allow us to apply Theorem \ref{steeltheorem} to the subspaces of $\CCC$ that we will consider in the next section.

\begin{lemma}\label{reasonably}
Let $X$ be a homogeneous Borel subspace of $\CCC$, and let $\mathbf{\Gamma}=[X]$. Assume that $X\notin\Delta$. Then $\mathbf{\Gamma}$ is reasonably closed.
\end{lemma}
\begin{proof}
By Lemma \ref{nonselfdual}, we see that $\mathbf{\Gamma}\in\{\mathbf{\Gamma}_u,\check{\mathbf{\Gamma}}_u\}$ for some $u\in D_0$. It follows from \cite[Lemma 4.2.17]{vanengelent} that $\mathbf{\Gamma}$ is reasonably closed.
\end{proof}

\newpage

\begin{theorem}[Steel]\label{steeltheorem}
Let $\mathbf{\Gamma}$ be a reasonably closed class of Borel subsets of $\CCC$.
Let $X$ and $Y$ be subspaces of $\CCC$ that satisfy the following conditions.
\begin{itemize}
\item $X,Y$ are everywhere properly $\mathbf{\Gamma}$.
\item $X,Y$ are either both first category or both Baire.
\end{itemize}
Then there exists a homeomorphism $h:\CCC\longrightarrow \CCC$ such that $h[X]=Y$.
\end{theorem}

The following is \cite[Lemma 2.7]{vanengeleni}.

\begin{theorem}[van Engelen]\label{steelconsequence}
Let $X,Y$ be homogeneous Borel subspaces of $\CCC$ such that $X,Y\notin\Delta$. Then the following conditions are equivalent.
\begin{itemize}
\item $X\approx Y$.
\item $X\equiv_W Y$ and $X,Y$ are either both first category or both Baire.
\end{itemize}
\end{theorem}
\begin{proof}
One implication is an immediate consequence of Lemma \ref{wadgeindependent}. Now assume that $X\equiv_W Y$ and $X,Y$ are either both first category or both Baire. First notice that, given any two dense in itself subspaces $A$ and $B$ of $\CCC$ and a Wadge reduction $f:\CCC\longrightarrow \CCC$ from $A$ to $B$, the restriction $f\re\cl(A):\cl(A)\longrightarrow\cl(B)$ will still be a Wadge reduction from $A$ to $B$, where $\cl$ denotes closure in $\CCC$. Therefore, we can assume that $X$ and $Y$ are dense in $\CCC$.

Let $\mathbf{\Gamma}=[X]=[Y]$. Notice that $\mathbf{\Gamma}$ is reasonably closed by Lemma \ref{reasonably}. Therefore, in order to to apply Theorem \ref{steeltheorem}, it will be sufficient to show that $X$ and $Y$ are everywhere properly $\mathbf{\Gamma}$. We will only prove this for $X$, since the proof for $Y$ is similar.

Let $U$ be a non-empty clopen subset of $\CCC$. It is trivial to show that $U\cap X\leq_W X$. Therefore $U\cap X\in\mathbf{\Gamma}$.
Now assume, in order to get a contradiction, that $U\cap X\in\check{\mathbf{\Gamma}}$. By Theorem \ref{hhomogeneouslyembedded},
we can fix a homeomorphism $h_U:\CCC\longrightarrow U$ such that $h_U[X]=U\cap X$. Since $h_U$ clearly witnesses that $X\leq_W U\cap X$,
we see that $X\in\check{\mathbf{\Gamma}}$, which contradicts the fact that $\mathbf{\Gamma}$ is non-self-dual (see Lemma \ref{nonselfdual}).
In conclusion, we have proved that $U\cap X\in\mathbf{\Gamma}\setminus\check{\mathbf{\Gamma}}$ for every non-empty clopen subset $U$ of $\CCC$.

Finally, assume that $U$ is a non-empty open subset of $\CCC$. Without loss of genrality, assume that $U$ is not clopen, and let $U_n$ for $n\in\omega$ be pairwise disjoint non-empty clopen subsets of $\CCC$ such that $\bigcup_{n\in\omega}U_n=U$.
Notice that each $U_n\cap X\in\mathbf{\Gamma}\setminus\check{\mathbf{\Gamma}}$. In particular, we immediately see that $U\cap X\notin\check{\mathbf{\Gamma}}$.
Now observe that $\SU(\mathbf{\Sigma}^0_1,\mathbf{\Gamma})=\mathbf{\Gamma}$ by Lemma \ref{separatedunions}. Since $U\cap X=\bigcup_{n\in\omega}(U_n\cap X)\in \SU(\mathbf{\Sigma}^0_1,\mathbf{\Gamma})$, this concludes the proof.
\end{proof}

\section{From homogeneous space to semifilter: the case above $\Delta$}

In this section, we will complete the proof of Theorem \ref{main}.

\begin{theorem}\label{notindelta}
Let $X$ be a zero-dimensional homogeneous Borel space such that $X\notin\Delta$. Then $X$ is homeomorphic to a semifilter.
\end{theorem}
\begin{proof}
Without loss of generality, assume that $X$ is a dense (Borel) subspace of $\CCC$. First assume that $X$ is Baire. In this case, by Theorem \ref{bairesemifilter}, there exists a Baire semifilter $\Ss$ such that $\Ss\equiv_W X$. It follows from Theorem \ref{steelconsequence} and Corollary \ref{semifilterhomogeneous} that $\Ss\approx X$.

Now assume that $X$ is first category. In this case, by Theorem \ref{firstcategorysemiideal}, there exists a first category semiideal $\RR$ such that $\RR\equiv_W X$. It follows from Theorem \ref{steelconsequence} and Corollary \ref{semifilterhomogeneous} that $\RR\approx X$. Let $\Ss=c[\RR]$, and observe that $\Ss$ is a semifilter. The trivial fact that $\Ss\approx\RR$ concludes the proof.
\end{proof}

For the proof of Theorem \ref{firstcategorysemiideal}, which is taken almost verbatim from the proof of \cite[Lemma 3.3]{vanengeleni}, we will need a few more preliminaries.
Given a space $Z$, it is possible to define the \emph{relative Wadge class} $\mathbf{\Gamma}_u(Z)$ for $u\in D$ by performing the same operations as in \cite[Definition 4.2.2]{vanengelent} to subsets of $Z$.
The following two lemmas (see \cite[Lemma 2.3]{vanengeleni} and \cite[Lemma 4.2.15]{vanengelent}) are then proved by a tedious but straightforward induction.
\begin{lemma}[van Engelen]\label{relativewadgeinverse}
Fix $u\in D$. Let $X$ and $Y$ be spaces. If $A\in\mathbf{\Gamma}_u(Y)$ and $f:X\longrightarrow Y$ is a continuous function then $f^{-1}[A]\in\mathbf{\Gamma}_u(X)$
\end{lemma}
\begin{lemma}[van Engelen]\label{relativewadgesubspace}
Fix $u\in D$. Suppose $X$ is a subspace of $\CCC$ and $A\subseteq X$. Then $A\in\mathbf{\Gamma}_u(X)$ if and only if $A=A'\cap X$ for some $A'\in\mathbf{\Gamma}_u$.
\end{lemma}

\begin{theorem}\label{firstcategorysemiideal}
Let $X$ be a homogeneous dense Borel subspace of $\CCC$. Assume that $X$ is first category and $X\notin\Delta$.
Then there exists a first category semiideal $\RR$ such that $\RR\equiv_W X$.
\end{theorem}

\begin{proof}
We begin by making the following definitions.
\begin{itemize}
\item $\DDD=\PP(2^{<\omega})$.
\item $K=\{\{x\re n:n\in\omega\}:x\in\CCC\}$.
\item Given $S\subseteq K$, let $\hat{S}=\{y\in\DDD:y\subseteq z\text{ for some }z\in S\}$.
\item Given $S\subseteq K$, let $\ddot{S}=\{y\in\hat{S}:y\text{ is infinite}\}$.
\item Given $e\in\Fin(2^{<\omega})$, let $\DDD_e=\{x\in\DDD:x\cap e=\varnothing\}$.
\end{itemize}
It is easy to check that $K$ and $\hat{K}$ are subspaces of $\DDD$ that are homeomorphic to $\CCC$, and that $\ddot{K}$ is a $\mathbf{\Pi}^0_2$ subset of $\DDD$. 

From this point on, assume that a bijection between $\omega$ and $2^{<\omega}$ has been fixed. This will allow us to identify $\CCC$ and $\DDD$. In particular, given any $X\subseteq\DDD$, it will make sense to consider the Wadge class generated by $X$.

By Lemma \ref{wadgeindependent}, we can assume without loss of generality that $X$ is a dense subspace of $K$. Let $\mathbf{\Gamma}=[X]$, and observe that $\mathbf{\Gamma}\in\{\mathbf{\Gamma}_u,\check{\mathbf{\Gamma}}_u\}$ for some $u\in D_0$ by Lemma \ref{nonselfdual}. Define $Y=\hat{X}$.

\textbf{Claim 1}: $Y\equiv_W X$.

The fact that $X\leq_W Y$ follows immediately from Lemma \ref{closedwadge}, because $X=Y\cap K$.
Now consider the unique function $f:\ddot{K}\longrightarrow K$ such that $x\subseteq f(x)$ for every $x\in\ddot{K}$,
and observe that $f$ is continuous. By Lemma \ref{relativewadgeinverse}, it follows that $\ddot{X}=f^{-1}[X]\in\mathbf{\Gamma}(\ddot{K})$. Hence, by Lemma \ref{relativewadgesubspace}, there exists
$X'\in\mathbf{\Gamma}$ such that $\ddot{X}=X'\cap\ddot{K}$. Since $Y=(X'\cap\ddot{K})\cup\Fin(2^{<\omega})$,
it follows from Lemma \ref{GdeltaFsigma} that $Y\in\mathbf{\Gamma}$, which completes the proof of Claim 1.

Now define
$$
\RR=\{y\cup e:y\in Y\text{ and }e\in\Fin(2^{<\omega})\},
$$
and observe that $\RR$ is a semiideal on $2^{<\omega}$.

\textbf{Claim 2}: $\RR\equiv_W X$.

Given $e\in\Fin(2^{<\omega})$, define $\RR_e=\RR\cap\hat{K}\cap\DDD_e=Y\cap\DDD_e$,
and observe that each $\RR_e$ is a semifilter on $2^{<\omega}\setminus e$.
In particular, each $\RR_e$ is homogeneous (by Corollary \ref{semifilterhomogeneous}) and not locally compact.
Notice that each $\RR_e\leq_W Y$, because $\RR_e$ is the intersection of $Y$ with the clopen subset $\DDD_e$ of $\DDD$.
By Claim 1, it follows that each $\RR_e\in\mathbf{\Gamma}$.

Given $e\in\Fin(2^{<\omega})$, define $\psi_e:\DDD\longrightarrow\DDD$ by $\psi_e(x)=x\cup e$,
and observe that $\psi_e\re\RR:\RR\longrightarrow\RR$ is continuous and closed.
Since $\psi_e\re\DDD_e$ is injective, it follows that $\psi_e[\RR_e]\approx\RR_e$.
Let $\{Z_n:n\in\omega\}$ be an enumeration of $\{\psi_e[\RR_e]:e\in\Fin(2^{<\omega})\}$,
and notice that each $Z_n\in\mathbf{\Gamma}$ by Lemma \ref{wadgeindependent}.
Furthermore, it is clear that $\RR=\bigcup_{n\in\omega}Z_n$.

Define $W_n=\cl(Z_n)\setminus\bigcup_{m<n}\cl(Z_m)$ for $n\in\omega$, where $\cl$ denotes closure in $\DDD$,
and observe that each $W_n$ is a $\mathbf{\Sigma}^0_2$ subset of $\DDD$.
Furthermore, using the fact that each $Z_n$ is closed in $\RR$, one can easily check that $\RR=\bigcup_{n\in\omega}(Z_n\cap W_n)$.
This shows that $\RR\in\SU(\mathbf{\Sigma}^0_2,\mathbf{\Gamma})$. Hence $\RR\in\mathbf{\Gamma}$ by Lemma \ref{separatedunions}.

To complete the proof of Claim 2, it remains to show that $X\leq_W \RR$. By Claim 1, it will be enough to show that $Y\leq_W \RR$.
This follows from Lemma \ref{closedwadge}, since $Y=\RR\cap\hat{K}$ is a closed subset of $\RR$ and $[Y]=\mathbf{\Gamma}$ by Claim 1.

\textbf{Claim 3}: $\RR$ is first category.

It is straightforward to check that $\RR$ has the finite union property (that is, $\bigcup F\notin\Cof(2^{<\omega})$ whenever $F\subseteq\RR$ is finite).
Therefore, $\RR\subseteq\II$ for some ideal $\II$ on $2^{<\omega}$. It follows easily from \cite[Theorem 8.47]{kechris} that $\RR$ is first category.
\end{proof}

\begin{corollary}\label{bairesemifilter}
Let $X$ be a homogeneous dense Borel subspace of $\CCC$. Assume that $X$ is Baire and $X\notin\Delta$.
Then there exists a Baire semifilter $\Ss$ such that $\Ss\equiv_W X$.
\end{corollary}
\begin{proof}
Notice that $\CCC\setminus X$ must also be dense in $\CCC$, otherwise $X$ would be locally compact.
Furthermore, since $X$ is Borel and Baire, it has a complete dense subspace $G$.
Since $G$ must be a dense $\mathbf{\Pi}^0_2$ subset of $\CCC$, it follows that $\CCC\setminus X$ is first category.
Also observe that $\CCC\setminus X$ is homogeneous by Corollary \ref{homogeneouscomplement}.
In conclusion, it is possible to apply Theorem \ref{firstcategorysemiideal}, which yields a first category semiideal $\RR$ such that $\RR\equiv_W\CCC\setminus X$.
Obviously, this implies that $\Ss\equiv_W X$, where $\Ss=\CCC\setminus\RR$. It is straightforward to check that $\Ss$ is a Baire semifilter.
\end{proof}

\section{Open problems}

The main open problem is of course whether the assumption ``Borel'' in Theorem \ref{main} can be weakened. First of all, we will show that it cannot be altogether removed. This is an immediate consequence of the following two propositions. The first one is trivial, and the second one follows from \cite[Proposition 8.3]{medinic} and Lemma \ref{pibase}. Recall that a \emph{$\lambda$-set} is a space in which every countable set is $\mathbf{\Pi}^0_2$. Observe that no subspace of a $\lambda$-set can be homeomorphic to $\CCC$.
\begin{proposition}
Let $\Ss$ be an uncountable semifilter. Then $\Ss$ contains a subspace that is homeomorphic to $\CCC$.
\end{proposition}
\begin{proof}
Since $\Ss$ is uncountable, we can fix $\Omega\in\Ss\setminus\Cof$. It is easy to realize that $\{x\subseteq\omega:\Omega\subseteq x\}$ is the desired subspace of $\Ss$.
\end{proof}
\begin{proposition}
There exists a homogeneous subspace of $\CCC$ of size $\omega_1$ that is a $\lambda$-set.
\end{proposition}
Furthermore, by a classical theorem of Martin and Solovay from \cite{martinsolovay} (see also \cite[Theorem 23.3]{miller} or \cite[Theorem 8.1]{medinizdomskyy}), it is consistent that every subspace of $\CCC$ of size $\omega_1$ is coanalytic.
Hence, it is consistent that Theorem \ref{main} fails when ``Borel'' is weakened to ``coanalytic''. However, we do not know the answer to the following questions. As in \cite[page 315]{kechris}, we will say that a space is \emph{projective} if it is homeomorphic to a projective subspace of some complete space (see \cite[Section 4]{medinizdomskyy} for a more detailed treatment).
For the definition of Projective Determinacy, see \cite[Definition 38.15]{kechris}.
\begin{question}
Can Theorem \ref{main} be generalized to analytic spaces in $\ZFC$?
\end{question}
\begin{question}\label{dethard}
Assume Projective Determinacy. Can Theorem \ref{main} be generalized to all projective spaces?
\end{question}

The assumption of Projective Determinacy is natural because it ensures that Theorem \ref{steeltheorem} generalizes to projective subspaces of $\CCC$ (see the original statement of the theorem in \cite{steel}).
In the next section, we will make the very first step towards answering Question \ref{dethard}.
On the other hand, answering Question \ref{dethard} in full using the same strategy as in this paper would require a detailed analysis, in the spirit of \cite{louveau}, of the Wadge classes generated by projective subsets of $\CCC$.
This problem, unsolved even at the lowest levels of the projective hierarchy, is the object of current research (see \cite{fournier}).
However, it might be possible to circumvent this obstacle by using a more direct approach.

Finally, since every semifilter is homogeneous by Corollary \ref{semifilterhomogeneous}, the following question seems natural.
\begin{question}
Is every semifilter strongly homogeneous?
\end{question}
We remark that the answer to the above question is affirmative for Borel semifilters (by Corollary \ref{borelsemifilterstronglyhomogeneous}) and for filters (by the following result).
\begin{proposition}\label{filterstronglyhomogeneous}
Let $\FF$ be a filter. Then $\FF$ is strongly homogeneous.
\end{proposition}
\begin{proof}
By Lemma \ref{pibase}, it will be enough to show that $\FF\cap [s]\approx\FF$ for every $s\in 2^{<\omega}$. So let $\ell\in\omega$ and $s:\ell\longrightarrow 2$. Since the case $\FF=\Cof$ is trivial, assume that $\FF\supsetneq\Cof$, and fix $\Omega\in\FF\setminus\Cof$ such that $\Omega\cap\ell=\varnothing$. Since $\Omega\notin\Cof$, we can fix a bijection $\pi:\omega\setminus\Omega\longrightarrow\omega\setminus(\ell\cup\Omega)$. Define the function $h:\CCC\longrightarrow [s]$ by setting
$$
\left.
\begin{array}{lcl}
& & h(x)(n)= \left\{
\begin{array}{ll}
s(n) & \textrm{if }n\in\ell,\\
x(n) & \textrm{if }n\in\Omega,\\
x(\pi^{-1}(n)) & \textrm{if }n\in\omega\setminus(\ell\cup\Omega).
\end{array}
\right.
\end{array}
\right.
$$
It is straightforward to check that $h$ is a homeomorphism. Furthermore, it is clear that $h(x)\cap\Omega= x\cap\Omega$ for every $x\in\CCC$. Since $\Omega\in\FF$, it follows that $h[\FF]=\FF\cap[s]$.
\end{proof}

\section{Analytic and coanalytic homogeneous spaces}

We will denote by $\mathbf{\Sigma}^1_1$ the collection of all analytic subsets of $\CCC$, and by $\mathbf{\Pi}^1_1$ the collection of all coanalytic subsets of $\CCC$.
We will say that a space is \emph{properly analytic} if it is analytic and not coanalytic. We will say that a space is \emph{properly coanalytic} if it is coanalytic and not analytic.
In this section we will show that, under $\mathbf{\Sigma}^1_1$-Determinacy (see \cite[Definition 26.3]{kechris}), Theorem \ref{main} extends to all analytic and coanalytic spaces.
This follows from Corollary \ref{determinacyextend}, together with the fact that every analytic coanalytic space is Borel (this is a classical result of Souslin, see \cite[Theorem 14.11]{kechris}).

Once again, our main tool will be \cite[Theorem 2]{steel}. Next, we state explicitly the version of this result that we will need. Notice that $\mathbf{\Sigma}^1_1$ and $\mathbf{\Pi}^1_1$ are reasonably closed by \cite[Proposition 14.4]{kechris}.
\begin{theorem}[Steel]\label{steeltheorem2}
Assume $\mathbf{\Sigma}^1_1$-Determinacy. Let $\mathbf{\Gamma}\in\{\mathbf{\Sigma}^1_1,\mathbf{\Pi}^1_1\}$.
Let $X$ and $Y$ be subspaces of $\CCC$ that satisfy the following conditions.
\begin{itemize}
\item $X,Y$ are everywhere properly $\mathbf{\Gamma}$.
\item $X,Y$ are either both first category or both Baire.
\end{itemize}
Then there exists a homeomorphism $h:\CCC\longrightarrow \CCC$ such that $h[X]=Y$.
\end{theorem}

\begin{lemma}\label{everywhere}
Let $\mathbf{\Gamma}\in\{\mathbf{\Sigma}^1_1,\mathbf{\Pi}^1_1\}$.
Assume that $X$ is a homogeneous dense subspace of $\CCC$ such that $[X]=\mathbf{\Gamma}$.
Then $X$ is everywhere properly $\mathbf{\Gamma}$.
\end{lemma}
\begin{proof}
Let $U$ be a non-empty clopen subset of $\CCC$. It is trivial to show that $U\cap X\leq_W X$. Therefore $U\cap X\in\mathbf{\Gamma}$. Using the fact that $\mathbf{\Gamma}$ is closed under countable unions (see \cite[Proposition 14.4]{kechris}),
one sees that $U\cap X\in\mathbf{\Gamma}$ for every non-empty open subset $U$ of $\CCC$.

Now assume, in order to get a contradiction, that $U$ is a non-empty open subset $U$ of $\CCC$ such that $U\cap X\in\check{\mathbf{\Gamma}}$.
Since $X$ is homogeneous, it is possible to find $X_n\subseteq\CCC$ for $n\in\omega$ such that each $X_n\approx U\cap X$ and $\bigcup_{n\in\omega}X_n=X$.
Since each $X_n\in\check{\mathbf{\Gamma}}$ and $\check{\mathbf{\Gamma}}$ is closed under countable unions, we obtain that $X\in\check{\mathbf{\Gamma}}$.
Using \cite[Theorem 14.11]{kechris}, it follows that $X$ is Borel, which contradicts the assumption $[X]=\mathbf{\Gamma}$.
\end{proof}

The following well-known result (see \cite[Theorems 3.2, 3.3, 3.4 and 3.6]{farkaskhomskiividnyanszky}) is the last ingredient that will be needed in the proof of Theorem \ref{determinacyfilter}.

\begin{proposition}\label{existsfilter}
Let $\mathbf{\Gamma}\in\{\mathbf{\Sigma}^1_1,\mathbf{\Pi}^1_1\}$. Then there exists a filter $\FF$ such that $[\FF]=\mathbf{\Gamma}$.
\end{proposition}

\begin{theorem}\label{determinacyfilter}
Assume $\mathbf{\Sigma}^1_1$-Determinacy. Let $X$ be a zero-dimensional homogeneous space. Assume that $X$ is properly analytic or properly coanalytic.
\begin{itemize}
\item If $X$ is first category then $X\approx\FF$ for some filter $\FF$.
\item If $X$ is Baire then $X\approx\CCC\setminus\FF$ for some filter $\FF$.
\end{itemize}
\end{theorem}
\begin{proof}
Assume without loss of generality that $X$ is a dense subspace of $\CCC$.
Let $\mathbf{\Gamma}=[X]$, and notice that $\mathbf{\Gamma}\in\{\mathbf{\Sigma}^1_1,\mathbf{\Pi}^1_1\}$ by \cite[Theorem 26.4]{kechris}.
It follows from Lemma \ref{everywhere} that $X$ is everywhere properly $\mathbf{\Gamma}$.

First assume that $X$ is first category. By Proposition \ref{existsfilter}, we can fix a filter $\FF$ such that $[\FF]=\mathbf{\Gamma}$.
It follows from Lemma \ref{everywhere} that $\FF$ is everywhere properly $\mathbf{\Gamma}$. Furthermore, using \cite[Theorem 21.6 and 8.47]{kechris}, it is easy to see that $\FF$ is first category.
Therefore, by Theorem \ref{steeltheorem2}, there exists a homeomorphism $h:\CCC\longrightarrow \CCC$ such that $h[X]=\FF$.

Now assume that $X$ is Baire. By Proposition \ref{existsfilter}, we can fix a filter $\FF$ such that $[\FF]=\check{\mathbf{\Gamma}}$.
Notice that $[\CCC\setminus\FF]=\mathbf{\Gamma}$. Furthermore, $\CCC\setminus\FF$ is homogeneous by Corollary \ref{finitemodificationshomogeneous}.
It follows from Lemma \ref{everywhere} that $\CCC\setminus\FF$ is everywhere properly $\mathbf{\Gamma}$. Furthermore, using the fact that $\FF$ is first category, it is easy to see that $\CCC\setminus\FF$ is Baire.
Therefore, by Theorem \ref{steeltheorem2}, there exists a homeomorphism $h:\CCC\longrightarrow \CCC$ such that $h[X]=\CCC\setminus\FF$.
\end{proof}

\begin{corollary}\label{determinacyextend}
Assume $\mathbf{\Sigma}^1_1$-Determinacy. Let $X$ be a zero-dimensional homogeneous space. Assume that $X$ is properly analytic or properly coanalytic.
Then $X$ is homeomorphic to a semifilter.
\end{corollary}
\begin{proof}
This is clear if $X$ is first category. If $X$ is Baire, apply Lemma \ref{semifiltercomplement}
\end{proof}

Finally, we remark that all the results in this section generalize in a straightforward way to the classes $\mathbf{\Sigma}^1_n$ and $\mathbf{\Pi}^1_n$ for every $n\geq 1$.

\end{document}